\documentclass{article}

\usepackage{amsmath}
\usepackage{amsfonts}
\usepackage{amsthm}
\usepackage{amssymb}
\usepackage{graphicx}
\usepackage{color}
\usepackage{verbatim} 
\usepackage{hyperref}
\usepackage{enumerate}
\usepackage{float}
\usepackage{tabulary}
\usepackage[utf8]{inputenc}
\usepackage{tikz-cd}

\theoremstyle{plain}
\newtheorem{theorem}{Theorem}

\newtheorem{lemma}[theorem]{Lemma}

\theoremstyle{definition}

\title{The 5-Way Scale}
\author{Tanya Khovanova \and Joshua Lee}

\begin{document}

\maketitle

\begin{abstract}
In this paper, we discuss coin-weighing problems that use a 5-way scale which has five different possible outcomes: MUCH LESS, LESS, EQUAL, MORE, and MUCH MORE. The 5-way scale provides more information than the regular 3-way scale. We study the problem of finding two fake coins from a pile of identically looking coins in a minimal number of weighings using a 5-way scale. We discuss similarities and differences between the 5-way and 3-way scale. We introduce a strategy for a 5-way scale that can find both counterfeit coins among $2^k$ coins in $k+1$ weighings, which is better than any strategy for a 3-way scale.
\end{abstract}

\section{Introduction}

Coin problems have been a fascination for mathematicians for a long time \cite{GN}. In most of them there are many coins that look identical and a balance scale with two pans. Usually all real coins weigh the same and all counterfeit coins weigh the same, though lighter than real coins. The clasical balance scale used in past models has three different outcomes: LESS, EQUAL, and MORE. If the left pan is lighter, we denote the outcome as LESS. If the right pan is lighter, we denote the outcome as MORE. If the two pans are of equal weights, the outcome is EQUAL.

In this paper we assume that the scale is more refined; that is, it has five different outcomes: MUCH LESS, LESS, EQUAL, MORE, and MUCH MORE. The LESS and MORE outcomes correspond to the weight difference between one fake coin and one real coin, while MUCH LESS and MUCH MORE to a larger difference. 

In the most classical problem \cite{GN} it is known that one coin is fake and lighter than the rest and the goal is to find it in the smallest number of weighings. For the new scale this problem is not interesting, as MUCH MORE or MUCH LESS outcomes are never possible with the same number of coins on both pans. In other words, the 5-way scale doesn't give any advantage; it can only be used as a 3-way scale. 

In this paper we study the problem of finding two counterfeit coins among many identical-looking coins. We introduce the problem in Section~\ref{sec:pf}.

In Section~\ref{sec:itb} we find a natural information-theoretic bound. In Section~\ref{sec:eq}, we study the first weighing and consider the outcome that gives us the least information. 

In Section~\ref{sec:separated} we study a special case when the two counterfeit coins are separated into two different piles. In Section~\ref{sec:twopilesstrategy} we give a strategy using $k+1$ weighings for the case of separated coins when we have two piles of equal sizes $2^k$. This allows to create a strategy for our general case in Section~\ref{sec:strat} which is better than any strategy for a 3-way scale: it uses $k+1$ weighings for $2^k$ coins.

\section{Problem Formulation}\label{sec:pf}

We assume that the coins are identical and there are two types of coins: real and counterfeit. Real coins weigh the same. Counterfeit coins also weigh the same, but lighter than real ones.

We consider a new type of scale with five possible outcomes for each weighing: MUCH LESS, LESS, EQUAL, MORE, and MUCH MORE. We put the same number of coins on the pans. LESS means that the number of counterfeit coins on the left pan is one more than the number of counterfeit coins on the right pan. MUCH LESS means that the left pan has at least two more counterfeit coins than the right pan. MORE and MUCH MORE are similarly defined. EQUAL means that there are an equal number of counterfeit coins on both pans.

As this scale is more precise than the usual 3-way scale in the standard coin problems, the number of coins that we can process in a given number of weighings is at least the same. Indeed, if we treat MUCH LESS as LESS and MUCH MORE as MORE, then the 5-way scale becomes a 3-way scale.

The problem of finding one counterfeit coin is not interesting, as we never can get MUCH LESS or MUCH MORE. Therefore, we assume that we have at least 2 counterfeit coins, and at least 2 real coins. We note that the case of $n$ counterfeit and $k$ real coins is equivalent to the case of $k$ counterfeit and $n$ real coins.

In this paper we study the simplest case of two counterfeit coins:

\begin{quote}
There are $N >3$ identical-looking coins, two of them are counterfeit. What is the smallest number of weighings so that we are guaranteed to find both counterfeit coins on a 5-way scale?
\end{quote}

\section{The Information Theoretic Bound}\label{sec:itb}

The standard way to find a lower bound for the number of coins that can be processed in a given number of weighings is to count outcomes. Here is the argument for a 3-way scale.

Each weighing has 3 outcomes. That means, there are a total of $3^w$ possible outcomes after $w$ weighings. We need to find our two counterfeit coins after these weighings. What is the maximum number of coins we can process?

Suppose we found our two coins. That means one of $3^w$ possible outcomes led us to these pair of bad counterfeit coins pretending to be real. From this observation follows an information theoretical bound of $\binom{N}{2} \leq 3^w$, as each outcome must point to a different pair of coins. That means $N(N+1) \leq 2\cdot 3^w$, or 
\begin{equation}\label{eq:ITB3}
N \leq \dfrac{\sqrt{8\cdot 3^w + 1}-1}{2} \approx \sqrt{2\cdot 3^w}.
\end{equation}

Computational results show that the bound is very close to the actual number of how many coins we can process for a 3-way scale \cite{Knop, KP, Li}. Given that we can ignore the distinction between MUCH LESS and LESS, any strategy that works for a 3-way scale, also works for a 5-way scale. 

But our scale is spiffier. Does that mean we can move the bound up? A similar argument replacing 3 with 5 shows that we can process not more than
\begin{equation*}\label{eq:ITB5}
N \leq \dfrac{\sqrt{8\cdot 5^w + 1}-1}{2} \approx \sqrt{2\cdot 5^w}
\end{equation*}
coins.

This is a larger number than the bound for the 3-way scale. But is our new bound tight? For the rest of the paper our goal is to find a strategy for a 5-way scale that is better than any strategy for a 3-way scale.

Let us consider an example. Suppose we have five coins with two counterfeits among them. How many weighings do we need? There are 10 possibilities for pairs of counterfeit coins. Our bound proves that one weighing is not enough. Can we do it in two weighings?

First compare two coins against two coins. If if the outcome is MUCH LESS or MUCH MORE, then, hooray, we found the two counterfeit coins. This is the beauty of the 5-way scale. There is no way to find two counterfeit coins out of five in one weighing using a regular 3-way scale.

Can our luck continue? If the first weighing is LESS or MORE, then we know that one counterfeit coin is outside, and the other is on the lighter pan. Indeed, comparing the two coins on the lighter pan would lead us to our second counterfeit coin.

If the first outcome is EQUAL, then we know that one counterfeit coin is on each pan, and one real coin must be outside. Denote the coins on the left side as 1 and 2, the right side as 3 and 4, and the coin outside as 5. In the second weighing we compare 1 and 3 with 2 and 5: if it is MUCH LESS, then our counterfeit coins are 1 and 3; if it is LESS, they are 1 and 4; if it is EQUAL, they are 2 and 3; if it is MORE, they are 2 and 4. Therefore, our strategy allows us to find our two counterfeit coins from five using only two weighings!

The total number of  possibilities for two counterfeit coins among 5 coins is 10. In contrast with our 5-way scale, we can't find the two counterfeit coins in two weighings using a 3-way scale.

Now we want to find a strategy that is better than the bound in Eq.~(\ref{eq:ITB3}) for any number of coins. We start by analyzing the first weighing.

\section{The First Weighing}\label{sec:eq}

Consider the first weighing where we compare $k$ coins against $k$ coins. We calculate the information produced by each outcome, that is, the number of possible pairs of fake coins. By symmetry, we only need to consider three cases.

\begin{enumerate}
\item If it is MUCH LESS, then both fake coins are on the left pan, corresponding to $\binom{k}{2}$ different possibilities. 
\item If it is LESS, then one fake coin is on the left pan and the other one is not on the scale, corresponding to $k(N-2k)$ possibilities.
\item If it is EQUAL, then either each pan has one fake coin or both fake coins are not on the scale, corresponding to $k^2 + \binom{N-2k}{2}$ possibilities.
\end{enumerate}

The maximum out of the three expressions above is provided by the last one as shown in the next lemma.

\begin{lemma}
$\max(\binom{k}{2},k(N-2k),k^2 + \binom{N-2k}{2}) = k^2 + \binom{N-2k}{2}$.
\end{lemma} 

\begin{proof}
The first value is not more than the third as   $\binom{k}{2} \leq k^2 + \binom{N-2k}{2}$. Now we want to show that the second value is not more than the third.

Both $N$ and $k$ are integers, and $\dfrac{1}{4} \leq \left(x-\dfrac{1}{2}\right)^2$ for any integer $x$, so 
\[\dfrac{1}{2} \leq \left(N-3k-\dfrac{1}{2}\right)^2 + \left(k-\dfrac{1}{2}\right)^2 \implies\]
\[0 \leq N^2 + 10k^2 - 6Nk - N + 2k \implies\]
\[2Nk - 6k^2 \leq N^2-4Nk+4k^2-N+2k \implies\]
\[Nk - 2k^2 \leq k^2 + \dfrac{(N-2k)(N-2k-1)}{2},\] as desired.
\end{proof}

The EQUAL outcome is the one that gives us the least information. But how much info?

Now we want to find the minimum possible value for Case~(3). 
We have that after the first weighing, if it is EQUAL, then the remaining possibilities are
\begin{equation}\label{eq:par}
k^2 + \dfrac{(N-2k)(N-2k-1)}{2},
\end{equation} where $k$ is the number of coins on each side of the scale. The next step is to make our life more complicated by assuming that $k$ is not an integer, but a real number. As a function of $k$, Eq.~(\ref{eq:par}) is a parabola and it reaches its minimum where the derivative, $6k - 2N + 1$, is zero. That is the minimum for  a real $k$ is at $k = \dfrac{N}{3}-\dfrac{1}{6}$. 

The real $k$ told us that we need to divide all the coins into approximately three equal parts. If $k \approx N/3$, then EQUAL leaves us with approximately $\dfrac{N^2-N}{6} = \dfrac{\binom{N}{2}}{3}$ bits of information. Though we have five different outcomes, with the first weighing we cannot divide the number of possibilities into five equal parts. In the best division, the largest pile of information is one third: We can't reduce the number of possibilities by more than a third.

This is unfortunate. In the worst case the 5-way scale behaves similar to a 3-way scale. If every weighing behaved like this then the final answer would be close to the bound in Eq.~(\ref{eq:ITB3}), and the 5-way scale wouldn't be much better than a 3-way scale.

Let's not despair yet and see how information is distributed if the first weighing has $N/3$ coins on each pan:

\begin{itemize}
    \item MUCH LESS or MUCH MORE: one pile of size $N/3$ with two counterfeit coins, about $N^2/18$ possibilities.
    \item LESS or MORE: two piles of size $N/3$ with one counterfeit coin each, $N^2/9$ possibilities.
    \item EQUAL: either one pile of size $N/3$ contains both coins, which corresponds to about $N^2/18$ possibilities or two other piles of size $N/3$ contain one fake coin each,  corresponding to $N^2/9$ possibilities. Summing up we get $N^2/6$ possibilities
\end{itemize}

We just discovered a  situation that requires our attention:

\begin{itemize}
    \item One fake coin in one pile and one fake coin in another pile.
\end{itemize}

We will call this case the case of \textit{separated coins} and sstudy it in the next section.

\section{Separated coins}\label{sec:separated}

Suppose one fake coin is in pile $P_1$ and the other in $P_2$, where $P_1$ and $P_2$ are disjoint. The two piles $P_1$ and $P_2$ are not assumed to have equal size. We denote the size of a pile $P$ as $|P|$.

We divide $P_1$ into three piles denoted $A_1$, $A_2$, and $A_3$, where $|A_1| = |A_2| = x|P_1|$, and $P_2$ into three piles 
$B_1$, $B_2$, and $B_3$, where $|B_1| = |B_2| = y|P_2|$. We are interested in how much we can reduce the number of possibilities $|P_1|\cdot|P_2|$ for the case two separated coins using one weighing.

In our weighing we compare $A_1B_1$ versus $A_2B_2$. From now on we write two or more piles together to indicate the union of those piles.

MUCH LESS means one fake coin is in pile $A_1$ and one in $B_1$, thus reducing the total number of possibilities by $xy$. 

LESS means that one coin is in $A_3$ and another in $B_1$, or one coin in $A_1$ and another in $B_3$. We have two disjoint cases, with possibilities reduced by $(1-2x)y$ and $x(1-2y)$ correspondingly.

EQUAL yields either one coin in $A_1$ and the other in $B_2$, or one coin in $A_2$ and the other in $B_1$, or one coin in $A_3$ and the other in $B_3$. We got three disjoint groups with reductions $xy$, $xy$ and $(1-2x)(1-2y)$.

Our $x$ and $y$ range between 0 and $\frac{1}{2}$. We want to find the minimum of 
\[\max\{xy,(1-2x)y+x(1-2y),2xy+(1-2x)(1-2y)\}.\]

We used a program to get us the answer we unfortunately expected: the minimum of $\frac{1}{3}$ is achieved for $x=y =\frac{1}{3}$. In the best case we can divide the number of possibilities by 3.

Again, the possibilities are not divided evenly and EQUAL is the worst case. 

If we had a 3-way scale, the strategy for the first weighing were the same. And in the worst case we would have divided our information by 3.

Are we doomed? Maybe we can't do better than a 3-way scale. Though we can divide the information into five groups, we can't divide it evenly. It seems that the worst case is the same as the worst case for the 3-way scale.

\section{Separated Coins Strategy}\label{sec:twopilesstrategy}

In this section we suggest a strategy for separated coins when two piles of coins are of the same size $N = 2^k$, which is a power of 2.

In this section we will encounter another case, which we call the \textit{doubly separated coins.} In this case we have four piles $A$, $B$, $C$, and $D$, all of size $2^k$. It is known that there is one coin in each of $A$ and $B$ or there is a coin in each of $C$ and $D$.

Let $f(k)$ be the maximum number of weighings in the strategy below to find both counterfeit coins in the separated case, where each pile has $2^k$ coins. Let $g(k)$ be the maximum number of weighings needed to find both counterfeit coins in the strategy below in the doubly separated case, where each pile has $2^k$ coins.

Let us calculate $f$ and $g$ for small $k$.

First, we observe that $f(0) = 0$: we do not need any weighings if the fake coins are separated into two piles of size 1 each. 

Next, let's calculate $f(1)$. 
We have four coins labeled 1, 2, 3, and 4, so that one fake coin is in the pile 12 and the other in the pile 34. We can find the fake coins in two weighings: first comparing 1 versus 2, then 3 versus 4. An exhaustive search shows that we can't find both coins in one weighings.
That means $f(1)=2$.

Next we want to calculate $g(0)$. In this case we have four coins total, and the two fake coins are either 1 and 2 or 3 and 4. In one weighing, comparing 12 versus 34, we can find the fake coins. That is, $g(0)=1$.

Now we are ready to show you a cool strategy for separated coins, where each of the piles is size $2^k$, where $k > 1$. We split the first pile containing one fake coin into four equal groups labeled as 1, 2, 3, and 4. We do the same with the second pile and labels 5, 6, 7, 8.

First weigh 12 vs 56. The result cannot be MUCH LESS or MUCH MORE.

\textbf{Case 1.} Suppose the first weighing shows LESS. This must mean that we are reduced to a problem where one fake coin is in 12, and the other is in 78. This is the case of separated coins with twice fewer coins. Thus, in this case we need  $f(k-1)+1$ weighings. The outcome MORE has similar results.

\textbf{Case 2.} If the scale shows EQUAL in the first weighing, then  either one coin is in 12 and the other is in 56, or that one coin is in 34 and the other is in 78. This corresponds to the case of doubly separated coins with $2^{k-1}$ coins in each pile: we need $g(k-1)+1$ weighings. 

This means 
\[f(k) = \max\{f(k-1)+1,g(k-1)+1\}.\]

Now consider the case of doubly separated coins, where each pile is of size $2^k$, and $k > 0$. We split each pile into two so that $A$ is 1 and 2, $B$ is 5 and 6, $C$ is 3 and 4, and $D$ is 7 and 8.
Now weigh 15 vs 37. 

\textbf{Case $1'$.} If the weighing is MUCH LESS, then one fake coin is in 1 and the other is in 5. We have a case of separated coins where each pile has $2^{k-1}$ coins: we need total of $f(k-1)+1$ weighings. Similarly, if the weighing is MUCH MORE, then one fake coin is in 3 and the other in 7.

\textbf{Case $2'$.} If the weighing is LESS, one coin is in 1 and one coin is in 6, or one coin is in 2 and the other is in 5. We have a case of doubly separated coins with piles half as large. Thus, we need $g(k-1)+1$ weighings.

\textbf{Case $3'$.} If the weighing is EQUAL, then 15 can't contain a fake coin. Therefore, either one coin is in 2 and the other is in 6, or one coin is in 4 and the other is in 8. Again, we have a case of doubly separated coins with piles half as large and need $g(k-1)+1$ weighings in total. It follows that 
\[g(k) = \max\{f(k-1)+1,g(k-1)+1\}.\]

By continuing this process, we either get to separated coins with each pile of size 2, or to doubly separated coins with each pile of size 1. Looking at the starting conditions of $f(1) = 2$ and $g(0) = 1$ and using induction, we get 
\[f(k) = k+1\quad \text{ and } \quad g(k) = k+1.\]

We produced a strategy that solves the separated coins case for the total of $2^w$ coins in $w$ weighings. This strategy is better than any strategy on a 3-way scale. 

Notice that the first weighing for separated coins is not very effective. Depending on the outputs LESS, MORE, or EQUAL, it divides information into 1/4, 1/4, and 1/2 respectively. But the worst case of EQUAL corresponds to the case of doubly separated coins. 

Let us see how the information is divided in the doubly separated case. We start with four piles $A$, $B$, $C$, and $D$ of size $2^k$. There is one fake coin in either $A$ and $B$, giving $2^{2k}$ possibilities, or $C$ and $D$, giving another $2^{2k}$ possibilities, the total being $2^{2k+1}$.

If the first weighing is MUCH LESS or MUCH MORE we get to the case of separated coins with pile sizes $2^{k-1}$, for a total of $2^{2k-2}$ possibilities each. Other weighing outcomes direct us back to doubly separated coins with sizes reduced by a factor of 2. That means they divide the information by 4. That is, five different outcomes divide the information into portions 1/8, 1/8, 1/4, 1/4, and 1/4. This is way better than a 3-way scale.

\section{The Strategy}\label{sec:strat}

Now that we have a strategy that works fast for separated coins, we want to extend it to the general case of 2 fake coins in a group of $2^k$ coins. Let $h(k)$ denote the maximum number of weighings necessary to find both coins in one pile of size $2^k$. We assume that $k > 1$.

Let us split the pile into four equal piles, labeling them 1, 2, 3, and 4. First weigh 1 versus 2.

\textbf{Case 1.} If the weighing is  LESS or MORE, we weigh 3 versus 4 to enter the scenario in Section~\ref{sec:twopilesstrategy} with four piles of size $2^{k-2}$ each. They can be processed in $k-1$ weighings. That means all given coins can be processed in $k+1$ weighings.

\textbf{Case 2.} 
If the weighing is MUCH LESS or MUCH MORE, our strategy allows us to find the coins in $h(k-2) +1$ weighings.

\textbf{Case 3.} 
If the first weighing is EQUAL, then either there is exactly one counterfeit coin in each of 1 and 2, or both coins are in 3 and 4. Next, we weigh 2 versus 3 to reduce the problem to the following cases
\begin{itemize}
    \item LESS---one counterfeit coin in each of 1 and 2. We need $f(k-2)+2$ weighings in total.
    \item EQUAL---both counterfeit coins are in 4. We need $h(k-2) +2$ weighings in total.
    \item MORE---one counterfeit coin in each of 3 and 4. We need $f(k-2)+2$ weighings in total.
    \item MUCH MORE---both counterfeit coins are in 3. We need $h(k-2) +2$ weighings in total.
\end{itemize}

That means after two weighings in Case 3, we either reduce the pile by a factor of 4, or split the half of the pile into two equal piles each containing one fake coin, which we know how to process by Section~\ref{sec:twopilesstrategy}. In particular, in the worst case we have 
\[h(k) = \max(f(k-2)+2,h(k-2)+2),\] where $f$ is defined in Section~\ref{sec:twopilesstrategy}. 

We can use an exhaustive search to see that $h(1) = 0$, and $h(2) = 2$. It follows that $h(3) = \max(f(1)+2,h(1)+2) = 4$. Similarly, $h(4) = \max(f(2)+2,h(2)+2) = 5$. Continuing by induction, we get that \[h(k) = k+1,\]
for $k > 2$.

Thus we can always solve a problem of finding two fake coins among $N = 2^w$ coins in $w+1$ weighings using a 5-way scale.

Let us see how the information is divided in the first weighing of our strategy. At the beginning we had $\binom{N}{2}$ or about  $N^2/2$ possibilities. If the first weighing is MUCH LESS or MUCH MORE the number of possibilities is divided by 16. If it is LESS or MORE, then one coin is in the pile of size $N/4$, and the other in the pile of size $N/2$. So the information is divided by about 4. So EQUAL is the worst case with 3/8 of the initial information. 

In the first weighing, we are not doing better than a 3-way scale. But later we get to the case of doubly separated coins, which allows us to become much more effective.

\section{Conclusion}

Working on this project was an emotional rollercoaster. We started with a small example of 5 coins that gave us hope that a 5-way scale was noticeably better than a 3-way scale. Then we started analyzing the first weighing in different situations and were extremely discouraged. It looked like a 5-way scale couldn't do better than a 3-way scale. We were so discouraged that for some time we were not looking into what the second weighing might do. That was a mistake, as later weighings make the strategy we found faster and better. Despite the fact that the 5-way scale can't divide the possibilities evenly, there is a strategy for a 5-way scale that is much better than any strategy with a 3-way scale. 

We gave our example of a fast strategy only for a specific number of coins $N = 2^k$. We leave it to the reader and future researchers to see if this strategy can be extended to any number of coins.

\section{Acknowledgments}

This project was part of the PRIMES-USA program. We are thankful to the program for allowing us the opportunity to conduct this research.

\end{document}